\documentclass[titlepage,12pt]{article} 
\usepackage{hyperref}
\usepackage{graphicx}
\usepackage[usenames,dvipsnames]{xcolor}
\usepackage{multirow}
\usepackage{amssymb,amsthm,amsmath} 
\usepackage{mathrsfs}
\usepackage[a4paper]{geometry}
\usepackage{datetime2}
\usepackage[utf8]{inputenc}
\usepackage[italian,english]{babel}

\date{}

\newcommand{\ep}{\varepsilon}
\newcommand{\re}{\mathbb{R}}

\newcommand{\ut}{\widehat{u}}
\newcommand{\ul}{u_{\lambda}}
\newcommand{\wl}{w_{\lambda}}
\newcommand{\hl}{h_{\lambda}}
\newcommand{\etal}{\eta_{\lambda}}
\newcommand{\rhol}{\rho_{\lambda}}
\newcommand{\thetal}{\theta_{\lambda}}

\newtheorem{thm}{Theorem}[section]

\newtheorem{rmk}[thm]{Remark}
\newtheorem{prop}[thm]{Proposition}

\newtheorem{lemma}[thm]{Lemma}
\newtheorem{open}[thm]{Open problem}

\title{Resonance effects for linear wave equations with scale invariant oscillating damping}

\author{Marina Ghisi\vspace{1ex}\\ 
{\normalsize Università degli Studi di Pisa} \\
{\normalsize Dipartimento di Matematica}\\ 
{\normalsize PISA (Italy)}\\
{\normalsize e-mail: \texttt{marina.ghisi@unipi.it}}
\and
Massimo Gobbino\vspace{1ex}\\ 
{\normalsize Università degli Studi di Pisa} \\
{\normalsize Dipartimento di Matematica}\\ 
{\normalsize PISA (Italy)}\\  
{\normalsize e-mail: \texttt{massimo.gobbino@unipi.it}}
}

\begin{document}
\maketitle

\begin{abstract}

We consider an abstract linear wave equation with a time-dependent dissipation that decays at infinity with the so-called scale invariant rate, which represents the critical case. We do not assume that the coefficient of the dissipation term is smooth, and we investigate the effect of its oscillations on the decay rate of solutions.

We prove a decay estimate that holds true regardless of the oscillations. Then we show that oscillations that are too fast have no effect on the decay rate, while oscillations that are in resonance with one of the frequencies of the elastic part can alter the decay rate. 

In the proof we first reduce ourselves to estimating the decay of solutions to a family of ordinary differential equations, then by using polar coordinates we obtain explicit formulae for the energy decay of these solutions, so that in the end the problem is reduced to the analysis of the asymptotic behavior of suitable oscillating integrals.

\vspace{6ex}

\noindent{\bf Mathematics Subject Classification 2020 (MSC2020):} 
35L20, 35L90, 35B40.

%35L20 Initial-boundary value problems for second-order hyperbolic equations
%35L90 Abstract hyperbolic equations
%35B40 - Asymptotic behavior of solutions to PDEs
%35L10 Second-order hyperbolic equations
%35L70 Second-order nonlinear hyperbolic equations
%35L72 Second-order quasilinear hyperbolic equations
%35B65 Smoothness and regularity of solutions
%35D30 Weak solutions to PDEs
		
\vspace{6ex}


\end{abstract}

%%%%%%%%%%%%%%%%%%%%%
%                   %
%   Inizio lavoro   %
%                   %
%%%%%%%%%%%%%%%%%%%%%
 
\section{Introduction}

Let $H$ be a real Hilbert space, and let $A$ be a non-negative self-adjoint operator on $H$ with dense domain $D(A)$. Let $t_{0}$ be a positive real number, and let $b:[t_{0},+\infty)\to[0,+\infty)$ be a function that we call damping coefficient. In this paper we consider the abstract damped wave equation 
\begin{equation}
u''(t)+b(t)u'(t)+Au(t)=0
\qquad
t\geq t_{0},
\label{eqn:main}
\end{equation}
with initial data
\begin{equation}
u(t_{0})=u_{0}\in D(A^{1/2}),
\qquad
u'(t_{0})=u_{1}\in H,
\label{eqn:data}
\end{equation}
and we investigate the effect of the damping coefficient $b(t)$ on the decay rate as $t\to +\infty$ of the classical energy of solutions
\begin{equation}
E_{u}(t)=|u'(t)|^{2}+|A^{1/2}u(t)|^{2}.
\nonumber
\end{equation}

\paragraph{\textmd{\textit{Some heuristics}}}

Thanks to the usual Fourier analysis, it is well-known that equation (\ref{eqn:main}) is equivalent to the family of ordinary differential equations
\begin{equation}
\ul''(t)+b(t)\ul'(t)+\lambda^{2}\ul(t)=0,
\label{ode:eqn}
\end{equation}
where $\lambda$ is a positive real parameter.

Let us assume for a while that $b(t)\equiv b_{0}$ is a positive real constant. In this case solutions to (\ref{ode:eqn}) can be explicitly computed, and two regimes appear.
\begin{itemize}

\item  If $\lambda$ is large with respect to $b_{0}$, and more precisely if $b_{0}^{2}-4\lambda^{2}<0$, then solutions can be written in the form
\begin{equation}
\ul(t)=c_{0}\exp\left(-\frac{b_{0}}{2}t\right)\sin\left(\frac{1}{2}\sqrt{4\lambda^{2}-b_{0}^{2}}\cdot t+\varphi_{0}\right)
\nonumber
\end{equation}
for suitable constants $c_{0}$ and $\varphi_{0}$ that depend on initial data. In particular, solutions oscillate, and the decay of their energy is given by
\begin{equation}
\ul'(t)^{2}+\lambda^{2}\ul(t)^{2}\sim\exp(-b_{0}t)
\qquad
\text{as }t\to +\infty.
\label{decay-ode-h}
\end{equation}

This is the \emph{oscillatory regime} or \emph{hyperbolic regime}, sometimes referred to as \emph{non-effective regime} after the classification introduced in~\cite{2006-JDE-Wirth-NE,2007-JDE-Wirth-E}.

\item  If $\lambda$ is small with respect to $b_{0}$, and more precisely if $b_{0}^{2}-4\lambda^{2}>0$, then solutions can be written in the form
\begin{equation}
\ul(t)=c_{1}\exp(-\mu_{1}t)+c_{2}\exp(-\mu_{2}t)
\nonumber
\end{equation}
where the constants $c_{1}$ and $c_{2}$ depend on initial data, while
\begin{equation}
\mu_{1}:=\frac{b_{0}+\sqrt{b_{0}^{2}-4\lambda^{2}}}{2}
\qquad\text{and}\qquad
\mu_{2}:=\frac{b_{0}-\sqrt{b_{0}^{2}-4\lambda^{2}}}{2}.
\nonumber
\end{equation}

In particular, solutions do not oscillate. Concerning decay rates, we observe that the second term is slower, and therefore the generic solution satisfies
\begin{equation}
\lambda^{2}\ul(t)^{2}\sim
\lambda^{2}\exp(-2\mu_{2}t)\leq
\lambda^{2}\exp\left(-\frac{2\lambda^{2}}{b_{0}}t\right)\leq
\frac{b_{0}}{t}=
\left[\frac{1}{b_{0}}\cdot t\right]^{-1},
\label{decay-ode-pu}
\end{equation}
and
\begin{equation}
\ul'(t)^{2}\sim
\mu_{2}^{2}\exp(-2\mu_{2}t)\leq
\frac{4\lambda^{4}}{b_{0}^{2}}\exp\left(-\frac{2\lambda^{2}}{b_{0}}t\right)\leq
\frac{2}{t^{2}}.
\label{decay-ode-pu'}
\end{equation}

We observe that the two terms in the energy have different decay rates, a feature that is typical of parabolic problems. 

This is the \emph{non-oscillatory regime} or \emph{parabolic regime}, sometimes referred to as  \emph{effective regime} after the classification introduced in~\cite{2006-JDE-Wirth-NE,2007-JDE-Wirth-E}.

\end{itemize}

We observe also that the decay rate in (\ref{decay-ode-h}) is optimal in the sense that all solutions decay with exactly that rate, while the decay rates in (\ref{decay-ode-pu}) and (\ref{decay-ode-pu'}) are optimal only when we consider the generic solution and we make the supremum with respect to $\lambda$.

When the damping coefficient depends on $t$, it is reasonable to expect that something similar happens. In particular, when $\lambda$ is large with respect to $b(t)$ one expects an oscillatory regime where the energy of solutions decays as
\begin{equation}
\exp\left(-\int_{t_{0}}^{t}b(s)\,ds\right),
\label{decay-h}
\end{equation}
while when $\lambda$ is small with respect to $b(t)$ one expects a non-oscillatory regime where the energy of solutions decays as
\begin{equation}
\left[\int_{t_{0}}^{t}\frac{1}{b(s)}\,ds\right]^{-1}.
\label{decay-p}
\end{equation}

Note that (\ref{decay-h}) is decreasing with respect to $b$, namely more damping yields more decay, while (\ref{decay-p}) is increasing, namely more damping yields less decay. The two expressions coincide when $b(t)=2/t$, in which case they provide the maximal decay rate of order $1/t^{2}$. More precisely, when $b(t)$ decays as $m/t$, then the value of the constant $m$ becomes essential, with $m=2$ being the threshold between the hyperbolic regime in which the decay is given by (\ref{decay-h}) and the parabolic regime in which the decay is given by (\ref{decay-p}). For this reason, the case where $b(t)\sim m/t$ represents the critical case.

\paragraph{\textmd{\textit{Previous literature}}}

When considering the wave equation of the form (\ref{eqn:main}), it is reasonable to expect that the decay rate of the energy of its solutions is the worst among the decay rates of its components, namely the minimum between (\ref{decay-h}) and (\ref{decay-p}). 

Results of this type have been proved since the 70s, starting with some of the model cases shown in Table~\ref{table:decay-rates}. Here we just mention the papers~\cite{1977-PJA-Matsumura,1980-Kyoto-Uesaka,2004-M2AS-Wirth} where the case $b(t)=m/t$ was considered, and the case in which $b(t)$ is a positive constant and the parabolic behavior is related to the so-called diffusion phenomenon (see~\cite{1976-Kyoto-Matsumura,1997-JDE-Nishihara,2011-JDE-RaduTodoYord}).

\begin{table}[ht]

\centering 
\setlength{\fboxsep}{1ex}
\setlength{\fboxrule}{0pt}
\begin{tabular}{|c||c|c|c|}
\cline{2-3}
\multicolumn{1}{c|}{} & 
\framebox{\textbf{Damping coefficient}} & \textbf{Decay rate of solutions}
\\
\cline{2-3}\hline
\multirow{5}{*}{\rotatebox{90}{\makebox[0pt]{\textbf{Oscillatory}}}} &
\framebox{$\displaystyle\int_{t_{0}}^{+\infty}b(s)\,ds<+\infty$} & No decay &
\multirow{6}{*}{$\displaystyle\exp\left(-\int_{t_{0}}^{t}b(s)\,ds\right)$}
\\
\cline{2-3}
& \framebox{$\dfrac{1}{t\log t}$} & $\dfrac{1}{\log t}$ &
\\
\cline{2-3}
& \framebox{$\dfrac{m}{t}$\quad (with $m\in(0,2)$)} & \framebox{$\dfrac{1}{t^{m}}$} &
\\
\hline\hline
\multirow{8}{*}{\rotatebox{90}{\makebox[0pt]{\textbf{Non-oscillatory}}}} & 
\framebox{$\dfrac{m}{t}$\quad (with $m\geq 2$)} & \framebox{$\dfrac{1}{t^{2}}$} &
\multirow{8}{*}{$\displaystyle\left[\int_{t_{0}}^{t}\frac{1}{b(s)}\,ds\right]^{-1}$}
\\
\cline{2-3}
& \framebox{$\dfrac{1}{t^{p}}$\quad (with $p\in(-1,1)$)} & $\dfrac{1}{t^{p+1}}$ &
\\
\cline{2-3}
& $t$ & \framebox{$\dfrac{1}{\log t}$} &
\\
\cline{2-3}
& \framebox{$\displaystyle\int_{t_{0}}^{+\infty}\frac{1}{b(s)}\,ds<+\infty$} &  No decay &
\\
\hline
\end{tabular}

\caption{decay rates corresponding to some model damping coefficients}
\label{table:decay-rates}

\end{table}

In the oscillatory regime the decay rates of Table~\ref{table:decay-rates} are optimal in the sense that all solutions decay with that rate, and there is also a scattering theory to solutions of the undamped equation (see~\cite{2004-M2AS-Wirth,2006-JDE-Wirth-NE}). In the non-oscillatory regime, the decay rates are determined at the low frequencies of the spectrum of $A$, and they are optimal in the sense that the square of the norm of the ``energy operator'', namely the quantity
\begin{equation}
\mathcal{E}(t):=\sup\left\{E_{u}(t):(u_{0},u_{1})\in D(A^{1/2})\times H,\ 
|u_{1}|^{2}+|A^{1/2}u_{0}|^{2}+|u_{0}|^{2}\leq 1\right\},
\label{defn:energy-norm}
\end{equation}
decays as prescribed, up to multiplicative constants.

In the last 20 years, starting with the papers~\cite{2006-JDE-Wirth-NE,2007-JDE-Wirth-E}, the results for the model cases have been progressively extended to more general classes of damping coefficients. This extension turned out to be a difficult problem, on which the progress is rather slow (see for example~\cite{2008-Hiroshima-Wirth,2008-JMAA-HirosawaWirth,2017-JMAA-Wirth,2021-JMAA-VargasDaLuz,2021-AsymptAn-Yamazaki,2023-AsymptAn-AslanEbert,2023-JEE-Sobajima}). 

As far as we know, almost all the results so far involve the following two types of assumptions on the damping coefficient.
\begin{itemize}

\item  Assumptions that control the effective or non-effective nature of the damping, namely prescribing on which side of the threshold $2/t$ the damping coefficient lies, so that it is clear which is smaller between (\ref{decay-h}) and (\ref{decay-p}). These assumptions usually involve the behavior of $t\cdot b(t)$ as $t\to +\infty$, the typical example being requiring that the limit is $+\infty$ for the effective regime, or that the limsup is strictly less than $2$ (and sometimes even less than $1$) for the non-effective regime.

\item  Assumptions that control the oscillations of the damping coefficient. These assumptions usually require that $b(t)$ is monotone and/or that its first derivative (and sometimes also some higher order derivatives) decays fast enough, or more generally that $b(t)=b_{1}(t)+b_{2}(t)$, where $b_{1}(t)$ is ``well-behaved'' in the previous sense, and $b_{2}(t)$, which carries the oscillations, is a lower order term with suitable integrability or stabilization properties.

\end{itemize}

Roughly speaking, the leitmotiv is that the result becomes more and more difficult both when the damping coefficient approaches the threshold $2/t$, and when large or fast oscillations are allowed. We refer to the introduction of the recent paper~\cite{2023-AsymptAn-AslanEbert} for a good summary of the previous literature. Here we limit ourselves to quoting four examples that have been considered in the past, and that could be useful for a better comparison with our results.

\begin{enumerate}

\item  (\cite[Example~3.1]{2008-JMAA-HirosawaWirth}) Solutions decay as prescribed by (\ref{decay-h}), namely as $1/t^{m}$, when
\begin{equation}
b(t):=\frac{m(1+\sin(t^{\alpha}))}{t}
\qquad\quad%\text{with }
m\in(0,1/2),
\quad
\alpha\in(0,1).
\label{hp:slow}
\end{equation}

In this case the damping coefficient falls into the non-effective and scale invariant regime (and actually it is far from the threshold $2/t$). Its oscillations have the same order as the principal part, but they are ``slow'' because $\alpha<1$.

\item  (\cite[Example~1]{2023-AsymptAn-AslanEbert}) Solutions decay as prescribed by (\ref{decay-h}), namely as $1/t^{m}$, when
\begin{equation}
\frac{m}{t}-\frac{1}{t\log^{\gamma}t}\leq b(t)\leq\frac{m}{t}+\frac{1}{t\log^{\gamma}t}
\qquad\quad%\text{with }
m\in(0,2),
\quad
\gamma>1.
\label{hp:abs-int}
\end{equation}

Also in this case the damping coefficient falls into the non-effective and scale invariant regime. Oscillations can be very fast, but they are a lower order term and, more important, this term is \emph{absolutely} integrable at infinity because of the condition $\gamma>1$.

\item  (\cite[Example~3]{2023-AsymptAn-AslanEbert}) Solutions decay as prescribed by (\ref{decay-p}), namely as $1/t^{2}$, when
\begin{equation}
\frac{m-r}{t}\leq b(t)\leq\frac{m+r}{t}
\qquad\quad%\text{with }
m>2,
\quad
r\ll m-2.
\nonumber
\end{equation}

In this case the damping coefficient falls into the effective and scale invariant regime. Fast oscillations of the same order are allowed, but their amplitude is required to be small.

\item  (\cite[Theorem~2.1]{2021-JMAA-VargasDaLuz}) Solutions decay as prescribed by (\ref{decay-p}), namely as $1/t^{p+1}$, when
\begin{equation}
\frac{m}{t^{p}}\leq b(t)\leq\frac{M}{t^{p}}
\qquad\quad%\text{with }
0<m\leq M,
\quad
p\in(-1,1).
\nonumber
\end{equation}

In this case oscillations are allowed to be fast and large, but $t\cdot b(t)\to +\infty$ and therefore we are not in the scale invariant regime.

\end{enumerate}

\paragraph{\textmd{\textit{Our contribution}}}

The aim of this paper is to consider equation (\ref{eqn:main}) with damping coefficients that decay at infinity with a scale invariant rate proportional to $1/t$, but neither lie on one precise side of the threshold $2/t$ between the effective and non-effective regime, nor satisfy regularity assumptions that limit their oscillations.

In the first result (see Theorem~\ref{thm:decay}) we consider any measurable damping coefficient that lies in between $m/t$ and $M/t$ for suitable constants $M\geq m> 0$. We prove that the decay rate of solutions is at least the worst between the rates prescribed by Table~\ref{table:decay-rates} in the two extreme cases $m/t$ and $M/t$, despite the potentially wild oscillations. We stress that we do not assume that $M$ and $m$ are on the same side with respect to 2. However, even in the special case where $b(t)$ lies in the effective regime, this result improves what was previously known (see Remark~\ref{rmk:improvement-p}). We suspect that a similar paradigm applies to larger ranges of oscillation, in the sense that whenever $b_{1}(t)\leq b(t)\leq b_{2}(t)$, where $b_{1}(t)$ and $b_{2}(t)$ are two well-behaved coefficients (for example those in Table~\ref{table:decay-rates}), then the decay rate of the energy of solutions to (\ref{eqn:main}) is at least the worst between the decay rates corresponding to $b_{1}(t)$ and $b_{2}(t)$ (see Open Problem~\ref{open}).

Then we focus on two examples that shed some light on the role of oscillations. In the second result (see Theorem~\ref{thm:alpha}) we consider a damping coefficient of the form
\begin{equation}
b(t):=\frac{a+r\sin(t^{\alpha})}{t}
\qquad
\forall t>0,
\label{defn:b-hyp-alpha}
\end{equation}
with $a\geq r>0$ and $\alpha>1$, and we show that the decay rate of solutions coincides with the one prescribed by Table~\ref{table:decay-rates} for $b(t)=a/t$. Roughly speaking, this suggests that the oscillations of the coefficient are too fast, so that some homogenization effect takes place in such a way that solutions do not see these oscillations. We recall that a similar phenomenon had already been observed in the case where $\alpha<1$, but in that case the oscillations were ineffective because they were too slow.

Finally, in the third result (see Theorem~\ref{thm:resonance}) we show the existence of a damping coefficient with a scale invariant behavior for which equation (\ref{eqn:main}) admits solutions that do not decay according to (\ref{decay-h}) or (\ref{decay-p}), but \emph{more slowly}. The construction of this damping coefficient is rather implicit, but a careful inspection of the argument reveals that it has a form similar to (\ref{defn:b-hyp-alpha}) with $\alpha=1$. The key point is that the oscillations of this damping coefficient have the same ``period'' as the solutions of the undamped version of (\ref{ode:eqn}) with a specific value of $\lambda$. This induces a \emph{resonance effect} between the free oscillations and the damping coefficient, and this resonance effect deteriorates the decay rate of the components of the solution corresponding to frequencies close to that specific value of $\lambda$.

As far as we know, this is the first example were solutions do not decay according to (\ref{decay-h}) or (\ref{decay-p}). A posteriori it justifies the difficulty in extending the results of Table~\ref{table:decay-rates} to less regular damping coefficients. Now we know that the extension is in general false, and for example the absolute integrability condition that appears in (\ref{hp:abs-int}) can not be replaced by simple integrability.

\paragraph{\textmd{\textit{Overview of the technique}}}

First of all, using Fourier analysis we reduce ourselves to proving $\lambda$-independent decay estimates for solutions to (\ref{ode:eqn}). To this end, we observe that the pair $(\ul(t),\ul'(t))$ can be written in the form
\begin{equation}
\ul(t)=\frac{1}{\lambda}\rhol(t)\cos(\thetal(t)),
\qquad\qquad
\ul'(t)=-\rhol(t)\sin(\thetal(t)),
\label{defn:polar}
\end{equation}
where $\rhol:[t_{0},+\infty)\to(0,+\infty)$ and $\thetal:[t_{0},+\infty)\to\re$ are solutions to the system of ordinary differential equations
\begin{eqnarray}
\rhol'(t) & = & -\rhol(t)b(t)\sin^{2}(\thetal(t))
\label{eqn:rho}
\\[0.5ex]
\thetal'(t) & = & \lambda-\frac{1}{2}b(t)\sin(2\thetal(t)).
\label{eqn:theta}
\end{eqnarray}

From the first equation it follows that the energy of the solution, namely the quantity
\begin{equation}
\rhol(t)^{2}=\ul'(t)^{2}+\lambda^{2}\ul(t)^{2},
\label{defn:rho}
\end{equation}
is given by
\begin{equation}
\rhol(t)^{2}=\rhol(t_{0})^{2}
\exp\left(-2\int_{t_{0}}^{t}b(s)\sin^{2}(\thetal(s))\,ds\right)
\qquad
\forall t\geq t_{0}.
\label{est:rho2}
\end{equation}

Now assume that $\lambda$ is large with respect to $b(t)$, which is always the case, at least for $t$ large enough, whenever $b(t)\to 0$ as $t\to +\infty$. In this hyperbolic regime, from equation (\ref{eqn:theta}) we can expect that $\thetal(t)\sim\lambda t$ and therefore it is reasonable to approximate the argument of the exponential in (\ref{est:rho2}) as
\begin{equation}
-2\int_{t_{0}}^{t}b(s)\sin^{2}(\thetal(s))\,ds\sim
-2\int_{t_{0}}^{t}b(s)\sin^{2}(\lambda s)\,ds.
\nonumber
\end{equation}

In this way the problem is reduced to estimating an oscillating integral, in which the oscillations of $b(s)$ might interact with the oscillations of $\sin^{2}(\lambda s)$. At this point three possible scenarios appear.
\begin{itemize}

\item  If the oscillations of $b(s)$ are ``slow'' compared with the oscillations of $\sin^{2}(\lambda s)$, then it is reasonable to replace the trigonometric term by its time-average, which is equal to 1/2. In this way we obtain that
\begin{equation}
-2\int_{t_{0}}^{t}b(s)\sin^{2}(\thetal(s))\,ds\sim
-\int_{t_{0}}^{t}b(s)\,ds,
\nonumber
\end{equation}
and therefore from (\ref{est:rho2}) we deduce that solutions decay as prescribed by (\ref{decay-h}).  

\item  If $b(s)$ contains terms whose oscillations are ``fast'' compared with the ones of $\sin^{2}(\lambda s)$, then these fast oscillations can be replaced by their time-average. For example, when $b(t)$ is given by (\ref{defn:b-hyp-alpha}), the term $\sin(s^{\alpha})$ oscillates faster because of the condition $\alpha>1$, and therefore it can be replaced by its time-average, which is equal to 0. Therefore, in this case we obtain that
\begin{equation}
-2\int_{t_{0}}^{t}b(s)\sin^{2}(\thetal(s))\,ds\sim
a\log\left(\frac{t_{0}}{t}\right)\sim
-\int_{t_{0}}^{t}b(s)\,ds,
\nonumber
\end{equation}
up to lower order terms, which again justifies an energy decay of the form (\ref{decay-h}).

\item  If $b(s)$ contains terms that oscillate as $\sin^{2}(\lambda s)$, then things are different. For example, in the proof of Theorem~\ref{thm:resonance} we construct a damping coefficient similar to (\ref{defn:b-hyp-alpha}), but with the term $\sin(t^{\alpha})$ replaced by something that behaves as $\cos(2\lambda t)=1-2\sin^{2}(\lambda t)$. With this choice we obtain that
\begin{equation}
-\int_{t_{0}}^{t}b(s)\,ds\sim 
-\int_{t_{0}}^{t}\frac{a+r\cos(2\lambda s)}{s}\,ds\sim
a\log\left(\frac{t_{0}}{t}\right),
\label{int1}
\end{equation}
and
\begin{equation}
-2\int_{t_{0}}^{t}b(s)\sin^{2}(\thetal(s))\,ds\sim
a\log\left(\frac{t_{0}}{t}\right)
-2r\int_{t_{0}}^{t}\frac{(1-2\sin^{2}(\lambda s))\sin^{2}(\lambda s)}{s}\,ds,
\label{int2}
\end{equation}
but now the last integral diverges with the same order of the first term, and therefore it can no longer be neglected. As a consequence, the exponentials of (\ref{int1}) and (\ref{int2}) have different orders, and hence the decay rate given by (\ref{est:rho2}) does not coincide with (\ref{decay-h}). 

\end{itemize}

Replacing oscillating integrals with their time-averages is the rough idea behind the proof of our main results.  Of course, a formal proof has to justify rigorously all the approximations, which we do in Propositions~\ref{prop:hyperbolic} and~\ref{prop:hyp-alpha}. More important, we need to consider also the parabolic regime in which $b(t)$ is large with respect to $\lambda$. We deal with this regime in Proposition~\ref{prop:parabolic}, where we use different (and somewhat more elementary) energy estimates, the main idea being that the parabolic regime applies just in a ``short'' time interval.

\paragraph{\textmd{\textit{Resonance effects in different models}}}

We conclude by mentioning some analogies with apparently different problems.

In~\cite{GGH-2016-SIAM} we considered again equation (\ref{eqn:main}), with an operator $A$ whose spectrum is either a finite set or an increasing sequence of positive real numbers (this assumption rules out the issue of low frequencies). Our aim was designing the damping coefficient $b(t)$ in such a way that all solutions to  (\ref{eqn:main}) decay as fast as possible.  We discovered that the best choice is a ``pulsating coefficient'' that alternates small and large values with a frequency that depends on the eigenvalues of $A$. In that model the resonance was exploited in order to produce a fast decay; here we exploit it in order to produce a decay that is slower than expected.

In~\cite{GGH-2019-Quantization} we considered a wave equation with a non-linear non-local damping, and we studied the asymptotic behavior of solutions. Again the key tool was the polar representation of solutions in the form (\ref{defn:polar}), which again led to the study of oscillating integrals, similar to the ones that appear in this paper, where some terms could be approximated by their time-averages.

Finally, we can not conclude without mentioning the related problem where the time-dependent coefficient is in front of the elastic term, namely the abstract wave equation
\begin{equation}
w''(t)+c(t)Aw(t)=0,
\nonumber
\end{equation}
and the corresponding family of ordinary differential equations
\begin{equation}
\wl''(t)+\lambda^{2}c(t)\wl(t)=0.
\label{ode:c(t)}
\end{equation}

It is well-known that, when $c(t)$ is a positive constant, the energy of solutions remains constant in time. On the contrary, when $c(t)$ is allowed to oscillate between two positive constants, then (\ref{ode:c(t)}) admits solutions that grow exponentially in time, the classical example being the case where
\begin{equation*}
c(t):=1-8\ep\sin(2\lambda t)-16\ep^{2}\sin^{4}(\lambda t),
\qquad
\wl(t):=\sin(\lambda t)\exp\left(2\ep\lambda t-\ep\sin(2\lambda t)\strut\right)
%\label{eqn:DGCDS}
\end{equation*}
for some small enough $\ep>0$. The existence of this anomalous growth was discovered in the seminal paper~\cite{DGCS}, and it has a lot of consequences both in terms of non-existence of solutions when the propagation speed and/or initial data are not regular enough (this problem has been intensively studied after~\cite{DGCS}, for more details we refer to the recent paper~\cite{2023-MathAnn-GG} and to the references quoted therein), and in terms of lack of the so-called generalized energy conservation when everything is smooth (see for example~\cite{2007-MathAnn-Hirosawa,2015-JMAA-EbeFitHir} and the references quoted therein).

Here we just point out that in the example mentioned above the time-dependent coefficient $c(t)$ and the solution $\wl(t)$ oscillate with the same period, and it is again a resonance effect that triggers the exponential growth of the energy.

\paragraph{\textmd{\textit{Structure of the paper}}}

This paper is organized as follows. In section~\ref{sec:statements} we fix the functional setting, and we state our results concerning the decay rate of solutions to (\ref{eqn:main}) and (\ref{ode:eqn}). In section~\ref{sec:osc-int}, which is the technical core of this paper, we study the convergence of some oscillating integrals. In section~\ref{sec:ODEs} we prove the key estimates for solutions of the family of ordinary differential  equations (\ref{ode:eqn}). Finally, in section~\ref{sec:PDEs} we prove the main results. 

%\clearpage

\setcounter{equation}{0}
\section{Statements}\label{sec:statements}

\paragraph{\textmd{\textit{Functional setting}}}

In this paper we assume that $H$ is a real Hilbert space and $A$ is a linear operator on $H$ with domain $D(A)$. We always assume that $A$ is unitary equivalent to a non-negative multiplication operator in some $L^{2}$ space. More precisely, we assume that there exist a measure space $(\mathcal{M},\mu)$, a measurable function $\lambda:\mathcal{M}\to[0,+\infty)$, and a linear isometry $\mathscr{F}:H\to L^{2}(\mathcal{M},\mu)$ with the property that for every $u\in H$ it turns out that
\begin{equation}
u\in D(A)
\quad\Longleftrightarrow\quad
\lambda(\xi)^{2}[\mathscr{F}u](\xi)\in L^{2}(\mathcal{M},\mu),
\nonumber
\end{equation}
and for every $u\in D(A)$ it turns out that
\begin{equation}
\left[\mathscr{F}(Au)\right](\xi)=
\lambda(\xi)^{2}[\mathscr{F}u](\xi)
\qquad
\forall\xi\in\mathcal{M}. 
\nonumber
\end{equation}

Roughly speaking, $\mathscr{F}$ is a sort of generalized Fourier transform that allows to identify every element $u\in H$ with a function $\ut\in L^{2}(\mathcal{M},\mu)$, and under this identification the operator $A$ becomes the multiplication operator by $\lambda(\xi)^{2}$ in $L^{2}(\mathcal{M},\mu)$. In particular, under this identification it turns out that
\begin{equation}
|u|^{2}_{H}=
\|\ut\|^{2}_{L^{2}(\mathcal{M},\mu)}=
\int_{\mathcal{M}}\ut(\xi)^{2}\,d\xi,
\label{defn:norm-xi}
\end{equation}
and more generally
\begin{equation}
|A^{\alpha}u|^{2}_{H}=
\int_{\mathcal{M}}\lambda(\xi)^{4\alpha}\cdot\ut(\xi)^{2}\,d\xi
\qquad
\forall\alpha>0,
\quad
\forall u\in D(A^{\alpha}),
\label{defn:A-alpha}
\end{equation}
where $D(A^{\alpha})$ is defined as the set of vectors $u\in H$ for which the integral in the right-hand side is finite.

At this point problem (\ref{eqn:main})--(\ref{eqn:data}) can be solved by considering, for every $\xi\in\mathcal{M}$, the function $\ut(t,\xi)$ that solves the ordinary differential equation
\begin{equation}
\ut\,''(t,\xi)+b(t)\ut\,'(t,\xi)+\lambda(\xi)^{2}\ut(t,\xi)=0
\label{eqn:uxi}
\end{equation}
(here ``primes'' denote derivatives with respect to time), with initial data
\begin{equation}
\ut(t_{0},\xi)=[\mathscr{F}u_{0}](\xi),
\qquad\qquad
\ut\,'(t_{0},\xi)=[\mathscr{F}u_{1}](\xi),
\label{data:xi}
\end{equation}
and finally setting $u(t):=\mathscr{F}^{-1}\ut(t,\xi)$. In this way one obtains that, if $b\in L^{1}((t_{0},T))$ for every $T>t_{0}$, then for every pair of initial data (\ref{eqn:data}) equation (\ref{eqn:main}) has a unique solution 
$$u\in C^{0}\left([t_{0},+\infty),D(A^{1/2})\right)\cap C^{1}\left([t_{0},+\infty),H\strut\right).$$

%\clearpage

\paragraph{\textmd{\textit{Main results}}}

Our first result concerns a non-regular damping coefficient that oscillates between two ``well-behaved'' scale invariant coefficients.

\begin{thm}[General oscillations]\label{thm:decay}

Let $H$ and $A$ be as in the functional setting described at the beginning of this section. 
Let $t_{0}$ be a positive real number, and let $b:[t_{0},+\infty)\to\re$ be a measurable function.

Let us assume that there exist two real numbers $M\geq m>0$ such that
\begin{equation}
\frac{m}{t}\leq b(t)\leq\frac{M}{t}
\qquad
\forall t\geq t_{0},
\label{hp:bound-m1m2}
\end{equation}
and let us set
\begin{equation}
\mu:=\min\{m,2\}.
\label{defn:mu}
\end{equation}

Then every solution to problem (\ref{eqn:main})--(\ref{eqn:data}) satisfies the decay estimate
\begin{equation}
|u'(t)|^{2}+|A^{1/2}u(t)|^{2}\leq
e^{m(M+8)}
\left(4|u_{1}|^{2}+|A^{1/2}u_{0}|^{2}+\frac{2}{t_{0}^{2}}|u_{0}|^{2}\right)
\left(\frac{t_{0}}{t}\right)^{\mu}
\label{th:main-decay}
\end{equation}
for every $t\geq t_{0}$.

\end{thm}

\begin{rmk}[Better decay for coercive operators]
\begin{em}

The decay estimate (\ref{th:main-decay}) is optimal because it is optimal when $b(t)=m/t$. However, we recall that in the effective regime, namely when $m>2$, the optimality is determined only at low frequencies, and what actually decays as $1/t^{2}$ is the quantity defined in (\ref{defn:energy-norm}). 

Things are different if the operator $A$ is coercive, namely if there exists $\lambda_{0}>0$ such that $|Au|\geq\lambda_{0}^{2}|u|$ for every $u\in D(A)$. In this case, under the same assumptions of Theorem~\ref{thm:decay}, all solutions satisfy
\begin{equation}
|u'(t)|^{2}+|A^{1/2}u(t)|^{2}\leq
\exp\left(\frac{m(M+8)}{\lambda_{0}t_{0}}\right)
\left(|u_{1}|^{2}+|A^{1/2}u_{0}|^{2}\right)
\left(\frac{t_{0}}{t}\right)^{m}
\hspace{1.5em}
\forall t\geq t_{0},
\nonumber
\end{equation}
namely all solutions decay with at least the hyperbolic rate $1/t^{m}$, even if $m>2$ (see Proposition~\ref{prop:hyperbolic}).

\end{em}
\end{rmk}

We suspect that, in the case $m<2$, estimate (\ref{th:main-decay}) might be true even if we allow much larger oscillations. More precisely, for the time being we have no counterexamples to the following question (note that in (\ref{hp:bound-open}) the damping coefficient is allowed to oscillate between two coefficients that yield the same decay rate of solutions according to Table~\ref{table:decay-rates}).

\begin{open}\label{open}
\begin{em}

Let $t_{0}$, $m$, $M$ be positive real numbers, with $m\in(0,2)$ and $M\geq mt_{0}^{m-2}$. Let $b:[t_{0},+\infty)\to\re$ be a measurable function such that
\begin{equation}
\frac{m}{t}\leq b(t)\leq\frac{M}{t^{m-1}}
\qquad
\forall t\geq t_{0}.
\label{hp:bound-open}
\end{equation}

Determine whether there exists a constant $\Gamma_{1}$, possibly depending on $t_{0}$, $m$, $M$, such that every solution to problem (\ref{eqn:main})--(\ref{eqn:data}) satisfies
\begin{equation}
|u'(t)|^{2}+|A^{1/2}u(t)|^{2}\leq
\Gamma_{1}\left(|u_{1}|^{2}+|A^{1/2}u_{0}|^{2}+|u_{0}|^{2}\right)
\left(\frac{t_{0}}{t}\right)^{m}
\qquad
\forall t\geq t_{0}.
\nonumber
\end{equation}

\end{em}
\end{open}

Our second main result concerns a damping coefficient with very fast oscillations.

\begin{thm}[Fast oscillations]\label{thm:alpha}

Let $H$ and $A$ be as in the functional setting described at the beginning of this section. 
Let us consider the damping coefficient $b(t)$ defined by (\ref{defn:b-hyp-alpha}), where $a$, $r$, $\alpha$ are three real numbers such that 
%Let $a$, $r$, $\alpha$ be three real numbers such that 
\begin{equation}
a>0,
\qquad\qquad
0\leq r\leq a,
\qquad\qquad
\alpha>1,
\label{hp:a-b-alpha}
\end{equation}
and let us set
\begin{gather}
\mu:=\min\{a,2\},
\qquad\qquad
B:=\frac{3r}{\alpha t_{0}^{\alpha}},
\label{defn:mu-B}
\\[0.5ex]
\Gamma_{2}:=\exp\left(a(a+r+8)+\frac{5r(a+r+4)}{2}+\frac{3r}{\alpha t_{0}^{\alpha}}+\frac{r\log 3}{\alpha-1}\right).
\label{defn:Gamma}
\end{gather}

Then every solution to problem (\ref{eqn:main})--(\ref{eqn:data}) satisfies the decay estimate
\begin{equation}
|u'(t)|^{2}+|A^{1/2}u(t)|^{2}\leq
\Gamma_{2}\left(4e^{2B}|u_{1}|^{2}+|A^{1/2}u_{0}|^{2}+\frac{2}{t_{0}^{2}}|u_{0}|^{2}\right)
\left(\frac{t_{0}}{t}\right)^{\mu}
\nonumber
\end{equation}
for every $t\geq t_{0}$.

\end{thm}

Our third result is an example in which the oscillations of the damping coefficient alter the expected decay rate of solutions.

\begin{thm}[Resonant oscillations]\label{thm:resonance}

Let $H$ and $A$ be as in the functional setting described at the beginning of this section, with $A$ not identically zero. 

Then for every pair of real numbers $a\geq r>0$ there exists a damping coefficient $b:[t_{0},+\infty)\to\re$ of class $C^{\infty}$ with the following properties.
\begin{enumerate}
\renewcommand{\labelenumi}{(\arabic{enumi})}

\item  (Scale invariant behavior). The damping coefficient $b$ satisfies
\begin{equation}
\frac{a-r}{t}\leq b(t)\leq\frac{a+r}{t}
\qquad
\forall t\geq t_{0}.
\label{th:bound-b}
\end{equation}

\item  (Integrability of oscillations). The limit
\begin{equation}
\lim_{t\to +\infty}\left(\frac{t_{0}}{t}\right)^{a}\exp\left(\int_{t_{0}}^{t}b(s)\,ds\right)
\label{th:lim-B}
\end{equation}
exists and is a real number.

\item  (Slower decay of solutions). There exists a positive real number $\Gamma_{3}$, that depends on $t_{0}$, $a$, $r$ and on the operator $A$, such that the function defined by (\ref{defn:energy-norm}) satisfies
\begin{equation}
\mathcal{E}(t)\geq
\frac{\Gamma_{3}}{t^{a-r/2}}
\qquad
\forall t\geq t_{0}.
\label{th:bad-decay}
\end{equation}

\end{enumerate}

\end{thm}

%\clearpage

We conclude by comparing our results with the previous examples that we mentioned in the introduction.

\begin{rmk}\label{rmk:improvement-p}
\begin{em}

When $m\geq 2$, Theorem~\ref{thm:decay} shows that solutions decay at least as $1/t^{2}$. In this special case our result improves \cite[Theorem~2]{2023-AsymptAn-AslanEbert}, both because $m$ can be equal to~2, and because the difference $M-m$ is not required to be small with respect to $m-2$. In other words, in this case solutions always decay as prescribed by (\ref{decay-p}), even if oscillations are large in size and are allowed to touch the critical threshold $2/t$.

Theorem~\ref{thm:alpha} is the counterpart of example (\ref{hp:slow}) in the range $\alpha>1$. Now we know that both slow and fast oscillations are ineffective, but for opposite reasons. In addition, in our result we do not need that oscillations remain within the non-effective regime.

Finally, let us consider Theorem~\ref{thm:resonance}. In the case where $a\in(0,2)$, it provides an example where (\ref{decay-p}) decays as $1/t^{2}$, (\ref{decay-h}) decays as $1/t^{a}$, but there are solutions to (\ref{eqn:main}) that decay at most as $1/t^{a-r/2}$. In particular, these solutions are slower than what prescribed and expected by (\ref{decay-p}) and (\ref{decay-h}). This shows that in~\cite[Theorem~1]{2023-AsymptAn-AslanEbert} an absolute integrability condition of the form (\ref{hp:abs-int}) can not be replaced by simple integrability.

\end{em}
\end{rmk}

\begin{rmk}[The classic model case]\label{rmk:classic}
\begin{em}

Let us consider the very special case where
$$
b(t)=\frac{a+r\sin t}{t}
\qquad
\forall t>0,
$$
for example with $a=1$ and $r=1/2$. This damping coefficient oscillates between the effective and the non-effective regime.

Theorem~\ref{thm:decay} applies with $m=1/2$ and $M=3/2$, yielding that solutions decay at least as $t^{-1/2}$. A refinement of our arguments, applied to this very special case, would give that actually solutions decay at least as $t^{-3/4}$, where $3/4=a-r/2$. For the sake of shortness, we do not include this computation in this paper.

On the contrary, Theorem~\ref{thm:resonance} does \emph{not} apply to this example. What we actually prove is the existence of a damping coefficient of the form 
$$
b(t)=\frac{a+r\sin(\eta(t))}{t}
\qquad
\forall t>0,
$$
even with $a=1$ and $r=1/2$, for which the decay rate is not better than $t^{-3/4}$, where again $3/4=a-r/2$. A careful inspection of the proof (where we have $\cos(2\eta(t))$ instead of $\sin(\eta(t))$, but the difference is not relevant) reveals that we can choose $\eta(t)$ such that $\eta(t)=t+O(\log t)$ as $t\to +\infty$, but we can not guarantee that $\eta(t)$ can be chosen to be exactly equal to $t$. This would require sharper estimates on some oscillating integrals. 

\end{em}
\end{rmk}

%\clearpage

\paragraph{\textmd{\textit{The key tool}}}

The proof of our main results relies on some estimates for the decay of the energy of solutions to the family of ordinary differential equations (\ref{ode:eqn}). We collect these estimates in the following proposition, whose three statements correspond to our three main results.

\begin{prop}\label{prop:main}

Let $t_{0}$ be a positive real number.

\begin{enumerate}
\renewcommand{\labelenumi}{(\arabic{enumi})}

\item  Let $\lambda\geq 0$ be a real number, and let $b:[t_{0},+\infty)\to\re$ be a measurable function that satisfies (\ref{hp:bound-m1m2}) for suitable constants $M\geq m>0$.

Then every solution to equation (\ref{ode:eqn}) satisfies the decay estimate
\begin{equation}
\ul'(t)^{2}+\lambda^{2}\ul(t)^{2}\leq
e^{m(M+8)}
\left\{4\ul'(t_{0})^{2}+\left(\lambda^{2}+\frac{2}{t_{0}^{2}}\right)\ul(t_{0})^{2}\right\}
\left(\frac{t_{0}}{t}\right)^{\mu}
\label{th:prop-gen}
\end{equation}
for every $t\geq t_{0}$, where $\mu$ is defined by (\ref{defn:mu}).

\item  Let $\lambda\geq 0$ be a real number, and let $b(t)$ be given by (\ref{defn:b-hyp-alpha}) for suitable parameters $a$, $r$, $\alpha$ satisfying (\ref{hp:a-b-alpha}). Let us define $\mu$, $B$, $\Gamma_{2}$ as in (\ref{defn:mu-B}) and (\ref{defn:Gamma}).

Then every solution to equation (\ref{ode:eqn}) satisfies the decay estimate
\begin{equation}
\ul'(t)^{2}+\lambda^{2}\ul(t)^{2}\leq
\Gamma_{2}\left\{4e^{2B}\ul'(t_{0})^{2}+\left(\lambda^{2}+\frac{2}{t_{0}^{2}}\right)\ul(t_{0})^{2}\right\}
\left(\frac{t_{0}}{t}\right)^{\mu}
\label{th:prop-alpha}
\end{equation}
for every $t\geq t_{0}$.

\item  For every pair of real numbers $a\geq r>0$, and for every $\lambda>0$, there exist a damping coefficient $b:[t_{0},+\infty)\to\re$ and a positive real number $\Gamma_{3}$, that depends on $t_{0}$, $a$, $r$, $\lambda$, such that
\begin{itemize}

\item  the damping coefficient $b$ is of class $C^{\infty}$ and satisfies (\ref{th:bound-b}) and (\ref{th:lim-B}),

\item  the solution to (\ref{ode:eqn}) with initial data $u(t_{0})=0$ and $u'(t_{0})=1$ satisfies 
\begin{equation}
\ul'(t)^{2}+\lambda^{2}\ul(t)^{2}\geq
\Gamma_{3}\left(\frac{t_{0}}{t}\right)^{a-r/2}
\qquad
\forall t\geq t_{0}.
\label{th:resonance}
\end{equation}

\end{itemize}

\end{enumerate}

\end{prop}

\begin{rmk}
\begin{em}

If the operator $A$ admits at least one positive eigenvalue, then from statement~(3) of Proposition~\ref{prop:main} it is immediate that Theorem~\ref{thm:resonance} holds true with a stronger conclusion, namely existence of a solution (and not just a supremum over all solutions) that decays less than the right-hand side of (\ref{th:bad-decay}).
    
\end{em}    
\end{rmk}
%\clearpage

\setcounter{equation}{0}
\section{Oscillating integrals}\label{sec:osc-int}

In this section we collect all the result concerning integrals of real functions that we need in the sequel. The first one is a general tool for proving boundedness or convergence of oscillating integrals.

\begin{lemma}\label{lemma:int-phi-psi}

Let $t_{0}$ be a positive real number, let $\varphi:[t_{0},+\infty)\to\re$ be a function of class $C^{2}$, and let $\psi:[t_{0},+\infty)\to\re$ be a function of class $C^{1}$.

Let us assume that $\varphi''(t)\geq 0$ for every $t\geq t_{0}$, and that there exist two positive real numbers $\varphi_{0}$ and $\Psi_{0}$ such that
\begin{equation}
|\varphi'(t)|\geq\varphi_{0}
\qquad\text{and}\qquad
|\psi'(t)|\leq\frac{\Psi_{0}}{t}
\qquad
\forall t\geq t_{0}.
\label{hp:phi-psi-0}
\end{equation}

Then  it turns out that
\begin{equation}
\left|\int_{t_{0}}^{t}\frac{\cos(\varphi(s))\sin(\psi(s))}{s}\,ds\right|\leq
\frac{4+\Psi_{0}}{\varphi_{0}t_{0}}
\qquad
\forall t\geq t_{0},
\label{th:int-phi-psi}
\end{equation}
and the following limit
\begin{equation}
\lim_{t\to +\infty}\int_{t_{0}}^{t}\frac{\cos(\varphi(s))\sin(\psi(s))}{s}\,ds
\label{th:lim-int-phi-psi}
\end{equation}
exists and is a real number.

\end{lemma}

\begin{proof}

Let us write the integral in the form
\begin{equation}
\int_{t_{0}}^{t}\frac{\cos(\varphi(s))\sin(\psi(s))}{s}\,ds=
\int_{t_{0}}^{t}\varphi'(s)\cos(\varphi(s))\cdot\frac{\sin(\psi(s))}{s\varphi'(s)}\,ds.
\nonumber
\end{equation}

Integrating by parts we obtain that
\begin{equation}
\int_{t_{0}}^{t}\frac{\cos(\varphi(s))\sin(\psi(s))}{s}\,ds=
I_{1}(t)+I_{2}(t)+I_{3}(t)+I_{4}(t),
\label{eqn:I1234}
\end{equation}
where
\begin{gather*}
I_{1}(t):=
\frac{\sin(\varphi(t))\sin(\psi(t))}{t\varphi'(t)}-
\frac{\sin(\varphi(t_{0}))\sin(\psi(t_{0}))}{t_{0}\varphi'(t_{0})},
\\[0.5ex]
I_{2}(t):=-\int_{t_{0}}^{t}\sin(\varphi(s))\cdot\frac{\cos(\psi(s))\psi'(s)}{s\varphi'(s)}\,ds,
\\[0.5ex]
I_{3}(t):=\int_{t_{0}}^{t}\sin(\varphi(s))\cdot\frac{\sin(\psi(s))}{s^{2}\varphi'(s)}\,ds,
\\[0.5ex]
I_{4}(t):=\int_{t_{0}}^{t}\sin(\varphi(s))\cdot\frac{\sin(\psi(s))}{s}\cdot
\frac{\varphi''(s)}{[\varphi'(s)]^{2}}\,ds.
\end{gather*}

Thanks to our assumption (\ref{hp:phi-psi-0}) we can estimate the first three terns as
\begin{gather*}
|I_{1}(t)|\leq
\frac{1}{\varphi_{0}}\left(\frac{1}{t}+\frac{1}{t_{0}}\right)\leq
\frac{2}{\varphi_{0}t_{0}},
\\[1ex]
|I_{2}(t)|\leq\frac{\Psi_{0}}{\varphi_{0}}\int_{t_{0}}^{t}\frac{ds}{s^{2}}\leq\frac{\Psi_{0}}{\varphi_{0}t_{0}},
\qquad\quad
|I_{3}(t)|\leq\frac{1}{\varphi_{0}}\int_{t_{0}}^{t}\frac{ds}{s^{2}}\leq\frac{1}{\varphi_{0}t_{0}},
\end{gather*}
and, since $\varphi''$ is nonnegative, we can estimate the last term as
\begin{equation}
|I_{4}(t)|\leq
\frac{1}{t_{0}}\int_{t_{0}}^{t}\frac{\varphi''(s)}{[\varphi'(s)]^{2}}\,ds=
\frac{1}{t_{0}}\left(\frac{1}{\varphi'(t_{0})}-\frac{1}{\varphi'(t)}\right)\leq
\frac{1}{\varphi_{0}t_{0}}.
\nonumber
\end{equation}

Plugging all these inequalities into (\ref{eqn:I1234}) we deduce (\ref{th:int-phi-psi}).

The same estimates show that $I_{1}(t)$ has a finite limit as $t\to +\infty$, and that the integrals $I_{2}(t)$, $I_{3}(t)$ and $I_{4}(t)$ are absolutely convergent, which is enough to prove that the limit in (\ref{th:lim-int-phi-psi}) exists and is a real number.
\end{proof}

\begin{rmk}\label{rmk:lemma-variant}
\begin{em}

Let us mention two variants of Lemma~\ref{lemma:int-phi-psi} that we exploit in the sequel (the proof is the same).
\begin{itemize}

\item  The same conclusions hold true with any combination of $\cos$/$\sin$ in the numerator of the fractions that we integrate in (\ref{th:int-phi-psi}) and (\ref{th:lim-int-phi-psi}).

\item  If we assume that both the inequality $\varphi''(t)\geq 0$, and the two inequalities in (\ref{hp:phi-psi-0}), hold true only in some finite interval $[t_{0},T_{0}]$, then we can conclude that the inequality in (\ref{th:int-phi-psi}) holds true for every $t$ in the same interval $[t_{0},T_{0}]$.

\end{itemize}

\end{em}
\end{rmk}

%\clearpage

In the following two results we apply Lemma~\ref{lemma:int-phi-psi} to the oscillating integrals that appear when we compute the decay rate of solutions to (\ref{ode:eqn}).

\begin{lemma}\label{lemma:osc-int}

Let $H_{0}$, $\lambda$, $t_{0}$ be three positive real numbers, and let $h:[t_{0},+\infty)\to\re$ be a function of class $C^{1}$ such that
\begin{equation}
|h'(t)|\leq\frac{H_{0}}{t}
\qquad
\forall t\geq t_{0}.
\label{hp:h'}
\end{equation}

Then for every positive integer $n$ it turns out that
\begin{equation}
\left|\int_{t_{0}}^{t}\frac{\cos(n\lambda s+nh(s))}{s}\,ds\right|\leq
\frac{2(H_{0}+4)}{\lambda t_{0}}
\qquad
\forall t\geq t_{0},
\label{th:osc-int-bound}
\end{equation}
and the following limit
\begin{equation}
\lim_{t\to +\infty}\int_{t_{0}}^{t}\frac{\cos(n\lambda s+nh(s))}{s}\,ds
\nonumber
\end{equation}
exists and is a real number.

\end{lemma}

\begin{proof}

Let us set
\begin{equation}
\psi_{1}(t):=\cos(n\lambda s)\cdot\cos(nh(s))
\qquad\text{and}\qquad
\psi_{2}(t):=\sin(n\lambda s)\cdot\sin(nh(s)),
\nonumber
\end{equation}
so that
\begin{equation}
\int_{t_{0}}^{t}\frac{\cos(n\lambda s+nh(s))}{s}\,ds=
\int_{t_{0}}^{t}\frac{\psi_{1}(s)}{s}\,ds-
\int_{t_{0}}^{t}\frac{\psi_{2}(s)}{s}\,ds.
\label{eqn:int-psi-12}
\end{equation}

Both integrals in the right-hand side fit into the framework of Lemma~\ref{lemma:int-phi-psi} and Remark~\ref{rmk:lemma-variant} with
\begin{equation}
\varphi(t):=n\lambda t,
\qquad
\psi(t):=nh(s),
\qquad
\varphi_{0}:=n\lambda,
\qquad
\Psi_{0}:= nH_{0}.
\nonumber
\end{equation}

Therefore, from Lemma~\ref{lemma:int-phi-psi} we deduce both the estimates
\begin{equation}
\left|\int_{t_{0}}^{t}\frac{\psi_{i}(s)}{s}\,ds\right|\leq
\frac{4+nH_{0}}{n\lambda t_{0}}\leq
\frac{H_{0}+4}{\lambda t_{0}}
\qquad
\forall t\geq t_{0},
\quad
\forall i=1,2,
\nonumber
\end{equation}
and the existence of the limit as $t\to +\infty$ of the two integrals in the right-hand side of (\ref{eqn:int-psi-12}). This completes the proof.
\end{proof}

%\clearpage

\begin{lemma}\label{lemma:s-alpha}

Let $H_{0}$, $\lambda$, $t_{0}$ be three positive real numbers, and let $h:[t_{0},+\infty)\to\re$ be a function of class $C^{1}$ satisfying (\ref{hp:h'}).

Then it turns out that
\begin{equation}
\left|\int_{t_{0}}^{t}\frac{\sin(s^{\alpha})\cdot\cos(2\lambda s+2h(s))}{s}\,ds\right|\leq
\frac{5(H_{0}+2)}{\lambda t_{0}}+\frac{\log 3}{\alpha-1}
\qquad
\forall t\geq t_{0}.
\nonumber
\end{equation}

\end{lemma}

\begin{proof}

Thanks to the classical product-to-sum and sum-to-product identities, we can write the numerator of the integrand in the form
\begin{equation}
\sin(s^{\alpha})\cos(2\lambda s+2h(s))=
\frac{1}{2}\left\{g_{1}(s)-g_{2}(s)+g_{3}(s)+g_{4}(s)\right\},	
\nonumber
\end{equation}
where
\begin{gather*}
g_{1}(s):=\cos(s^{\alpha}+2\lambda s)\sin(2h(s)),
\qquad
g_{2}(s):=\cos(s^{\alpha}-2\lambda s)\sin(2h(s)),
\\
g_{3}(s):=\sin(s^{\alpha}+2\lambda s)\cos(2h(s)),
\qquad
g_{4}(s):=\sin(s^{\alpha}-2\lambda s)\cos(2h(s)).
\end{gather*}

Therefore, it is enough to show that
\begin{equation}
\left|\int_{t_{0}}^{t}\frac{g_{i}(s)}{s}\,ds\right|\leq
\begin{cases}
\dfrac{H_{0}+2}{\lambda t_{0}}      & \text{if $i=1,3$}, 
\\[3ex]
\dfrac{4(H_{0}+2)}{\lambda t_{0}}+\dfrac{\log 3}{\alpha-1}\quad     & \text{if $i=2,4$}.
\end{cases}
\label{est:int-gi}
\end{equation}

In the cases $i=1$ and $i=3$ we apply Lemma~\ref{lemma:int-phi-psi} and Remark~\ref{rmk:lemma-variant} with
\begin{equation}
\varphi(s):=s^{\alpha}+2\lambda s,
\qquad
\psi(s):=2h(s),
\qquad
\varphi_{0}:=2\lambda,
\qquad
\Psi_{0}:=2H_{0},
\nonumber
\end{equation}
and we deduce that
\begin{equation}
\left|\int_{t_{0}}^{t}\frac{g_{i}(s)}{s}\,ds\right|\leq
\frac{2H_{0}+4}{2\lambda t_{0}}=
\frac{H_{0}+2}{\lambda t_{0}}
\qquad
\forall i=1,3.
\nonumber
\end{equation}

In the cases $i=2$ and $i=4$ we would like to apply Lemma~\ref{lemma:int-phi-psi} and Remark~\ref{rmk:lemma-variant} with
\begin{equation}
\varphi(s):=s^{\alpha}-2\lambda s,
\qquad
\psi(s):=2h(s),
\qquad
\varphi_{0}:=\lambda,
\qquad
\Psi_{0}:=2H_{0}.
\label{choices:phi-psi}
\end{equation}

The problem is that the first inequality in (\ref{hp:phi-psi-0}) is not necessarily satisfied for every $t\geq t_{0}$. In order to overcome this difficulty, we consider the two times $0<t_{1}<t_{2}$ such that
\begin{equation}
\alpha t_{1}^{\alpha-1}=\lambda
\qquad\text{and}\qquad
\alpha t_{2}^{\alpha-1}=3\lambda,
\nonumber
\end{equation}
and we observe that
\begin{equation}
\varphi'(s)\leq -\lambda
\qquad
\forall s\in(0,t_{1}]
\qquad\quad\text{and}\quad\qquad
\varphi'(s)\geq \lambda
\qquad
\forall s\geq t_{2}.
\nonumber
\end{equation}

Let us consider now any interval $[t_{3},t_{4}]\subseteq[t_{0},+\infty)$. If either $[t_{3},t_{4}]\subseteq[t_{0},t_{1}]$ or $[t_{3},t_{4}]\subseteq[t_{2},+\infty)$, then we can apply Lemma~\ref{lemma:int-phi-psi} in the interval $[t_{3},t_{4}]$ with the choices (\ref{choices:phi-psi}), and deduce that
\begin{equation}
\left|\int_{t_{3}}^{t_{4}}\frac{g_{i}(s)}{s}\,ds\right|\leq
\frac{2(H_{0}+2)}{\lambda t_{3}}\leq
\frac{2(H_{0}+2)}{\lambda t_{0}}
\qquad
\forall i=2,4.
\label{est:g2-1}
\end{equation}

If $[t_{3},t_{4}]\subseteq[t_{1},t_{2}]$, then we obtain that
\begin{equation}
\left|\int_{t_{3}}^{t_{4}}\frac{g_{i}(s)}{s}\,ds\right|\leq
\int_{t_{3}}^{t_{4}}\frac{ds}{s}\leq
\int_{t_{1}}^{t_{2}}\frac{ds}{s}=
\log\left(\frac{t_{2}}{t_{1}}\right)=
\frac{\log 3}{\alpha-1}
\qquad
\forall i=2,4.
\label{est:g2-2}
\end{equation}

Finally we set
\begin{equation}
J_{1}:=[t_{0},t]\cap[t_{0},t_{1}],
\qquad
J_{2}:=[t_{0},t]\cap[t_{1},t_{2}],
\qquad
J_{3}:=[t_{0},t]\cap[t_{2},+\infty),
\nonumber
\end{equation}
and we write the integral of $g_{i}(s)/s$ over $[t_{0},t]$ as the sum of the integrals over $J_{1}$, $J_{2}$, $J_{3}$ (depending on the position of $t_{0}$ and $t$ with respect to $t_{1}$ and $t_{2}$, one or two of the $J_{k}$'s might be empty or just a singleton). We observe that the integrals over $J_{1}$ and $J_{3}$ satisfy (\ref{est:g2-1}), while the integral over $J_{2}$ satisfies (\ref{est:g2-2}). Summing the three estimates we obtain exactly (\ref{est:int-gi}) for $i=2$ and $i=4$.
\end{proof}

%\clearpage

The last result that we need is an estimate from above for the function
\begin{equation}
\gamma(m,t_{0},t):=\int_{t_{0}}^{t}\left(\frac{t_{0}}{s}\right)^{m}ds
\qquad
\forall t\geq t_{0}.
\label{defn:gamma}
\end{equation}

\begin{lemma}

Let $m$ and $t_{0}$ be positive real numbers, and let $\mu$ be defined as in (\ref{defn:mu}).

Then the function defined by (\ref{defn:gamma}) satisfies
\begin{equation}
\left(\frac{t_{0}}{t}\right)^{2}\gamma(m,t_{0},t)^{2}\leq
t_{0}^{2}\left(\frac{t_{0}}{t}\right)^{\mu}
\qquad
\forall t\geq t_{0}.
\label{th:gamma}
\end{equation}

\end{lemma}

\begin{proof}

In the case $m\geq 2$ the required inequality reduces to $\gamma(m,t_{0},t)\leq t_{0}$, which is true because in this case
\begin{equation}
\gamma(m,t_{0},t)=
\int_{t_{0}}^{t}\left(\frac{t_{0}}{s}\right)^{m}\,ds\leq
\int_{t_{0}}^{t}\left(\frac{t_{0}}{s}\right)^{2}\,ds=
t_{0}^{2}\left(\frac{1}{t_{0}}-\frac{1}{t}\right)\leq
t_{0}.
\nonumber
\end{equation}

In the case $m\in(0,2)$, with the change of variable $\sigma:=t_{0}/s$ we obtain that
\begin{equation}
\gamma(m,t_{0},t)=t_{0}\int_{t_{0}/t}^{1}\frac{1}{\sigma^{2-m}}\,d\sigma,
\nonumber
\end{equation}
so that (\ref{th:gamma}) reduces to
\begin{equation}
\left(\frac{t_{0}}{t}\right)^{2-m}
\left[\int_{t_{0}/t}^{1}\frac{1}{\sigma^{2-m}}\,d\sigma\right]^{2}\leq 1
\qquad
\forall m\in(0,2),
\quad
\forall t\geq t_{0}.
\nonumber
\end{equation}

Setting $x:=t_{0}/t$ and $b:=1-m/2$, this is equivalent to proving that
\begin{equation}
\int_{x}^{1}\left(\frac{x}{\sigma^{2}}\right)^{b}d\sigma\leq 1
\qquad
\forall b\in(0,1),
\quad
\forall x\in(0,1).
\nonumber
\end{equation}

For every fixed $x\in(0,1)$, the left-hand side is a convex function of $b$, and hence it attains its maximum either in the limit as $b\to 0^{+}$, or  in the limit as $b\to 1^{-}$. Since both limits are equal to $1-x$, the inequality is proved.
\end{proof}

%\clearpage

\setcounter{equation}{0}
\section{Estimates for a family of ODEs}\label{sec:ODEs}

In the following two subsections we prove different types of estimates for solutions to the family of ordinary differential equations (\ref{ode:eqn}). These estimates hold true under rather general assumption on the damping coefficient, and are satisfied for all admissible values of $\lambda$ and $t$. The proof of Proposition~\ref{prop:main} follows in the third subsection from a combination of these estimates, the main idea being that we exploit the ``parabolic'' version when $b(t)$ is large with respect to $\lambda$, namely when $t$ is small, and the ``hyperbolic'' version when $b(t)$ is small with respect to $\lambda$, namely when $t$ is large enough.

\subsection{Estimates in the ``parabolic'' regime}

\begin{prop}[``Parabolic'' regime]\label{prop:parabolic}

Let $t_{0}$ be a positive real number, and let $b_{1}:[t_{0},+\infty)\to\re$ and $b_{2}:[t_{0},+\infty)\to\re$ be two measurable functions. 

Let us set $b(t):=b_{1}(t)+b_{2}(t)$, and let us assume that
\begin{enumerate}
\renewcommand{\labelenumi}{(\roman{enumi})}

\item  $b(t)\geq 0$ for every $t\geq t_{0}$,

\item  there exists a positive real number $m$ such that
\begin{equation}
b_{1}(t)\geq\frac{m}{t}
\qquad
\forall t\geq t_{0},
\label{hp:b1}
\end{equation}

\item  there exists a real number $B$ such that
\begin{equation}
\left|\int_{t_{0}}^{t}b_{2}(s)\,ds\right|\leq B
\qquad
\forall t\geq t_{0}.
\label{hp:b2}
\end{equation}

\end{enumerate}

Then for every $\lambda>0$, and for every solution to equation (\ref{ode:eqn}), there exists $t_{1}\geq t_{0}$ (that depends on $\lambda$ and on initial data) such that (we recall that $\gamma(m,t_{0},t)$ is the function defined by (\ref{defn:gamma})) 
\begin{itemize}

\item  for every $t\in[t_{0},t_{1}]$ the solution satisfies the estimate
\begin{equation}
\ul'(t)^{2}+\lambda^{2}\ul(t)^{2}\leq
2\lambda^{2}\ul(t_{0})^{2}+2e^{2B}\ul'(t_{0})^{2}
\left\{\left(\frac{t_{0}}{t}\right)^{2m}+\lambda^{2}\gamma(m,t_{0},t)^{2}\right\},
\label{th:P-before}
\end{equation}

\item for every $t\geq t_{1}$ the solution satisfies the estimate
\begin{equation}
\ul'(t)^{2}+\lambda^{2}\ul(t)^{2}\leq
2\lambda^{2}\ul(t_{0})^{2}+2e^{2B}\ul'(t_{0})^{2}\lambda^{2}\gamma(m,t_{0},t_{1})^{2}.
\label{th:P-after}
\end{equation}

\end{itemize}

\end{prop}

\begin{proof}

Let us write $\ul(t)$ in the form
\begin{equation}
\ul(t):=u_{\lambda,1}(t)+u_{\lambda,2}(t),
\nonumber
\end{equation}
where $u_{\lambda,1}$ is the solution to equation (\ref{ode:eqn}) with initial data $u_{\lambda,1}(t_{0})=\ul(t_{0})$ and $u_{\lambda,1}'(t_{0})=0$, while $u_{\lambda,2}$ is the solution to equation (\ref{ode:eqn}) with initial data $u_{\lambda,2}(t_{0})=0$ and $u_{\lambda,2}'(t_{0})=\ul'(t_{0})$. We observe that
\begin{equation}
\ul'(t)^{2}+\lambda^{2}\ul(t)^{2}\leq
2\left(u_{\lambda,1}'(t)^{2}+\lambda^{2}u_{\lambda,1}(t)^{2}\right)+
2\left(u_{\lambda,2}'(t)^{2}+\lambda^{2}u_{\lambda,2}(t)^{2}\right),
\label{est:u-u1-u2}
\end{equation}
so that in the sequel it is enough to estimate the energy of $u_{\lambda,1}$ and $u_{\lambda,2}$ separately. To this end, for $i=1,2$ we consider the energy
\begin{equation}
E_{i}(t):=u_{\lambda,i}'(t)^{2}+\lambda^{2}u_{\lambda,i}(t)^{2},
\nonumber
\end{equation}
and we observe that
\begin{equation}
E_{i}'(t)=-2b(t)u_{\lambda,i}'(t)^{2}\leq 0,
\qquad
\forall t\geq t_{0}
\quad
\forall i=1,2.
\label{est:Ei'}
\end{equation}

In the case of $u_{\lambda,1}$, this is enough to conclude that
\begin{equation}
u_{\lambda,1}'(t)^{2}+\lambda^{2}u_{\lambda,1}(t)^{2}=
E_{1}(t)\leq
E_{1}(t_{0})=
\lambda^{2}\ul(t_{0})^{2}
\qquad
\forall t\geq t_{0}.
\label{est:u1}
\end{equation}

In the case of $u_{\lambda,2}$ we assume, without loss of generality, that $\ul'(t_{0})>0$, and we define $t_{1}$ as the smallest real number $t\geq t_{0}$ such that $u_{\lambda,2}'(t)=0$. In the interval $[t_{0},t_{1})$ we know that $u_{\lambda,2}'(t)>0$, and hence also $u_{\lambda,2}(t)>0$. In particular, from (\ref{ode:eqn}) we obtain that
\begin{equation}
u_{\lambda,2}''(t)+b(t)u_{\lambda,2}'(t)=-\lambda^{2}u_{\lambda,2}(t)\leq 0
\qquad
\forall t\in[t_{0},t_{1}].
\nonumber
\end{equation}

Integrating this differential inequality we deduce that
\begin{equation}
0\leq u_{\lambda,2}'(t)\leq \ul'(t_{0})\exp\left(-\int_{t_{0}}^{t}b(s)\,ds\right)
\qquad
\forall t\in[t_{0},t_{1}].
\nonumber
\end{equation}

Now from assumptions (\ref{hp:b1}) and (\ref{hp:b2}) we obtain that
\begin{equation}
-\int_{t_{0}}^{t}b(s)\,ds=
-\int_{t_{0}}^{t}b_{1}(s)\,ds-
\int_{t_{0}}^{t}b_{2}(s)\,ds\leq
m\log\left(\frac{t_{0}}{t}\right)+B,
\nonumber
\end{equation}
and therefore
\begin{equation}
0\leq u_{\lambda,2}'(t)\leq
%\ul'(t_{0})\exp\left(-m\int_{t_{0}}^{t}\frac{1}{s}\,ds\right)\leq
\ul'(t_{0})\left(\frac{t_{0}}{t}\right)^{m}e^{B}
\qquad
\forall t\in[t_{0},t_{1}].
\label{est:u2'}
\end{equation}

Recalling that $u_{\lambda,2}(0)=0$, this implies also that
\begin{equation}
0\leq
u_{\lambda,2}(t)=
\int_{t_{0}}^{t}u_{\lambda,2}'(s)\,ds\leq
\ul'(t_{0})\gamma(m,t_{0},t)e^{B}
\qquad
\forall t\in[t_{0},t_{1}].
\label{est:u2}
\end{equation}

Plugging (\ref{est:u2'}), (\ref{est:u2}) and (\ref{est:u1}) into (\ref{est:u-u1-u2}) we obtain (\ref{th:P-before}) for every $t\in[t_{0},t_{1}]$.

For $t\geq t_{1}$ we consider the energy $E_{2}(t)$, and from (\ref{est:Ei'}) and (\ref{est:u2}) with $t=t_{1}$ we conclude that
\begin{equation}
u_{\lambda,2}'(t)^{2}+\lambda^{2}u_{\lambda,2}(t)^{2}=
E_{2}(t)\leq
E_{2}(t_{1})=
\lambda^{2}u_{\lambda,2}(t_{1})^{2}\leq
\lambda^{2}e^{2B}\ul'(t_{0})^{2}\gamma(m,t_{0},t_{1})^{2}
\nonumber
\end{equation}
for every $t\geq t_{1}$. Plugging this inequality and (\ref{est:u1}) into (\ref{est:u-u1-u2}) we obtain (\ref{th:P-before}) for every $t\geq t_{1}$.
\end{proof}

%\clearpage

\subsection{Estimates in the ``hyperbolic'' regime}

As announced in the introduction, the key tool is the polar representation of solutions to (\ref{ode:eqn}), which can be stated as follows (we omit the standard proof).

\begin{lemma}[Polar representation of solutions]\label{lemma:polar}

Let $t_{0}$ be a positive real number, and let $b:[t_{0},+\infty)\to\re$ be a continuous function.

Then every solution to equation (\ref{ode:eqn}) has the following properties.
\begin{enumerate}
\renewcommand{\labelenumi}{(\arabic{enumi})}

\item  The pair $(\ul(t),\ul'(t))$ can be written in the form (\ref{defn:polar}), where $\rhol:[t_{0},+\infty)\to(0,+\infty)$ and $\thetal:[t_{0},+\infty)\to\re$ are solutions to the system of ordinary differential equations (\ref{eqn:rho})--(\ref{eqn:theta}).

\item   The function $\thetal(t)$ can be written in the form 
\begin{equation}
\thetal(t)=\lambda t+\hl(t)
\label{eqn:theta-h}
\end{equation}
for a suitable function $\hl:[t_{0},+\infty)\to\re$ of class $C^1$ such that
\begin{equation}
|\hl'(t)|\leq \frac{1}{2}|b(t)|
\qquad
\forall t\geq t_{0}.
%\label{est:h'(t)}
\nonumber
\end{equation}

\item  The energy of the solution, namely the quantity (\ref{defn:rho}), is given by (\ref{est:rho2}).

\end{enumerate}

\end{lemma}

%\clearpage

\begin{prop}[``Hyperbolic'' regime -- General oscillations]\label{prop:hyperbolic}

Let $t_{0}$ be a positive real number, and let $b:[t_{0},+\infty)\to\re$ be a measurable function that satisfies (\ref{hp:bound-m1m2}) for suitable constants $M\geq m>0$.

Then for every $\lambda>0$ all solutions to equation (\ref{ode:eqn}) satisfy the decay estimate
\begin{equation}
\ul'(t)^{2}+\lambda^{2}\ul(t)^{2}\leq
\exp\left(\frac{m(M+8)}{\lambda t_{0}}\right)
\left(\ul'(t_{0})^{2}+\lambda^{2}\ul(t_{0})^{2}\right)\left(\frac{t_{0}}{t}\right)^{m}
\qquad
\forall t\geq t_{0}.
\label{th:hyperbolic}
\end{equation}

\end{prop}

\begin{proof}

With a classical approximation procedure, we can assume that the damping coefficient is continuous. In this case we write $\ul(t)$ and $\ul'(t)$ as in (\ref{defn:polar}), and we reduce ourselves to estimating from above the exponential in (\ref{est:rho2}). 

To this end, from the bound from below in (\ref{hp:bound-m1m2}) we deduce that
\begin{equation*}
-\int_{t_{0}}^{t}2b(s)\sin^{2}(\thetal(s))\,ds\leq
-\int_{t_{0}}^{t}\frac{2m\sin^{2}(\thetal(s))}{s}\,ds=
-\int_{t_{0}}^{t}\frac{m}{s}\,ds+
\int_{t_{0}}^{t}\frac{m\cos(2\thetal(s))}{s}\,ds.
\end{equation*}

In order to estimate the last integral, from statement~(2) of Lemma~\ref{lemma:polar} we know that $\thetal(t)$ can be written in the form (\ref{eqn:theta-h}) for a suitable $C^1$ function $\hl(t)$ that in this case satisfies
\begin{equation}
|\hl'(t)|\leq\frac{M}{2t}
\qquad
\forall t\geq t_{0}.
\nonumber
%\label{est:h'(t)}
\end{equation}
because of the bound from above in (\ref{hp:bound-m1m2}). Therefore, the integral fits into the framework of Lemma~\ref{lemma:osc-int} with $H_{0}:=M/2$ and $n=2$, from which we conclude that
\begin{equation}
-\int_{t_{0}}^{t}2b(s)\sin^{2}(\thetal(s))\,ds\leq
m\log\left(\frac{t_{0}}{t}\right)+\frac{m(M+8)}{\lambda t_{0}}.
\nonumber
\end{equation}

Plugging this estimate into (\ref{est:rho2}), and recalling (\ref{defn:rho}), we obtain exactly (\ref{th:hyperbolic}).
\end{proof}

%\clearpage

\begin{prop}[``Hyperbolic'' regime -- Fast oscillations]\label{prop:hyp-alpha}

Let $t_{0}$ be a positive real number, and let $b:[t_{0},+\infty)\to\re$ be the damping coefficient defined by (\ref{defn:b-hyp-alpha}) for suitable parameters $a$, $r$, $\alpha$ satisfying (\ref{hp:a-b-alpha}). 
%Let $\Gamma$ be the constant defined by (\ref{defn:Gamma}).

Then for every $\lambda>0$ all solutions to equation (\ref{ode:eqn}) satisfy the decay estimate
\begin{equation}
\ul'(t)^{2}+\lambda^{2}\ul(t)^{2}\leq
\Gamma_{4}\left(\ul'(t_{0})^{2}+\lambda^{2}\ul(t_{0})^{2}\right)\left(\frac{t_{0}}{t}\right)^{a}
\qquad
\forall t\geq t_{0},
\label{th:hyp-alpha}
\end{equation}
where
\begin{equation}
\Gamma_{4}:=\exp\left(\frac{2a(a+r+8)+5r(a+r+4)}{2\lambda t_{0}}+\frac{3r}{\alpha t_{0}^{\alpha}}+\frac{r\log 3}{\alpha-1}\right).
\label{defn:Gamma-bis}
\end{equation}

\end{prop}

\begin{proof}

As in the proof of Proposition~\ref{prop:hyperbolic} we write the solution in the form (\ref{defn:polar}), and we reduce ourselves to estimating from above the exponential in (\ref{est:rho2}). Moreover, again we obtain that $\thetal(t)$ can be written in the form (\ref{eqn:theta-h}) with $\hl(t)$ that in this case satisfies
\begin{equation}
|\hl'(t)|\leq\frac{a+r}{2t}
\qquad
\forall t\geq t_{0}.
\nonumber
\end{equation}

Now we observe that
\begin{equation}
-2\int_{t_{0}}^{t}b(s)\sin^{2}(\thetal(s))\,ds=
I_{1}(t)+I_{2}(t)+I_{3}(t)+I_{4}(4),
\label{defn:I1234}
\end{equation}
where
\begin{gather*}
I_{1}(t):=-\int_{t_{0}}^{t}\frac{a}{s}\,ds=a\log\left(\frac{t_{0}}{t}\right),
\qquad
I_{2}(t):=-r\int_{t_{0}}^{t}\frac{\sin(s^{\alpha})}{s}\,ds,
\\[1ex]
I_{3}(t):=a\int_{t_{0}}^{t}\frac{\cos(2\thetal(s))}{s}\,ds,
\qquad
I_{4}(t):=r\int_{t_{0}}^{t}\frac{\sin(s^{\alpha})\cos(2\thetal(s))}{s}\,ds.
\end{gather*}

Let us estimate the last three integrals. As for $I_{2}$, a classical integration by parts shows that
\begin{equation}
\int_{t_{0}}^{t}\frac{\sin(s^{\alpha})}{s}\,ds=
\frac{\cos(t_{0}^{\alpha})}{\alpha t_{0}^{\alpha}}-\frac{\cos(t^{\alpha})}{\alpha t^{\alpha}}-
\int_{t_{0}}^{t}\frac{\cos(s^{\alpha})}{s^{\alpha+1}}\,ds,
\nonumber
\end{equation}
from which we deduce that
\begin{equation}
|I_{2}(t)|\leq\frac{3r}{\alpha t_{0}^{\alpha}}
\qquad
\forall t\geq t_{0}.
\label{est:pre-b2}
\end{equation}

As for $I_{3}$, we apply Lemma~\ref{lemma:osc-int} with $H_{0}:=(a+r)/2$ and $n=2$, and we deduce that
\begin{equation}
|I_{3}(t)|\leq\frac{a(a+r+8)}{\lambda t_{0}}.
\nonumber
\end{equation}

As for $I_{4}$, we apply Lemma~\ref{lemma:s-alpha} with $H_{0}:=(a+r)/2$, and we deduce that
\begin{equation}
|I_{4}(t)|\leq
r\left(\frac{5(a+r+4)}{2\lambda t_{0}}+\frac{\log 3}{\alpha-1}\right).
\nonumber
\end{equation}

Plugging all these estimates into (\ref{defn:I1234}), and recalling (\ref{est:rho2}) and (\ref{defn:rho}), we obtain exactly (\ref{th:hyp-alpha}).
\end{proof}

%\clearpage

%\clearpage

%\clearpage

\subsection{Proof of Proposition~\ref{prop:main}}

\subsubsection*{Statement~(1)}

If $\lambda=0$ equation (\ref{ode:eqn}) can be explicitly integrated, and the result follows from the explicit formula for solutions. Therefore, in the sequel we assume that $\lambda$ is positive.

If $\lambda\geq 1/t_{0}$ we apply Proposition~\ref{prop:hyperbolic}, and from (\ref{th:hyperbolic}) we obtain that
\begin{equation}
\ul'(t)^{2}+\lambda^{2}\ul(t)^{2}\leq
e^{m(M+8)}\left\{\ul'(t_{0})^{2}+\lambda^{2}\ul(t_{0})^{2}\right\}\left(\frac{t_{0}}{t}\right)^{m}
\nonumber
\end{equation}
for every $t\geq t_{0}$, which is enough to establish (\ref{th:prop-gen}) in this case.

If $\lambda<1/t_{0}$ we start by applying Proposition~\ref{prop:parabolic} with
\begin{equation}
b_{1}(t):=b(t),
\qquad\qquad
b_{2}(t)\equiv 0,
\qquad\qquad
B:=0.
\nonumber
\end{equation}

To this end we divide the half-line $t\geq t_{0}$ into the three subsets
\begin{equation}
\left[t_{0},\min\left\{t_{1},\frac{1}{\lambda}\right\}\right],
\qquad\quad
\left[\min\left\{t_{1},\frac{1}{\lambda}\right\},\frac{1}{\lambda}\right],
\qquad\quad
\left[\frac{1}{\lambda},+\infty\right),
\label{division}
\end{equation}
where $t_{1}$ is the time provided by Proposition~\ref{prop:parabolic}.

In the first interval it turns out that $t_{0}\leq t\leq t_{1}$, and hence we can exploit estimate (\ref{th:P-before}), from which we obtain that
\begin{equation}
\ul'(t)^{2}+\lambda^{2}\ul(t)^{2}\leq
2\lambda^{2}\ul(t_{0})^{2}+2\ul'(t_{0})^{2}
\left\{\left(\frac{t_{0}}{t}\right)^{2m}+\lambda^{2}\gamma(m,t_{0},t)^{2}\right\}.
\nonumber
\end{equation}

Now in this first interval we know that $t\leq 1/\lambda$, namely $\lambda\leq 1/t$, and hence
\begin{equation}
\ul'(t)^{2}+\lambda^{2}\ul(t)^{2}\leq
\frac{2}{t_{0}^{2}}\left(\frac{t_{0}}{t}\right)^{2}\ul(t_{0})^{2}
+2\ul'(t_{0})^{2}\left(\frac{t_{0}}{t}\right)^{2m}
+\frac{2}{t_{0}^{2}}\ul'(t_{0})^{2}\left(\frac{t_{0}}{t}\right)^{2}\gamma(m,t_{0},t)^{2}.
\nonumber
\end{equation}

Recalling (\ref{th:gamma}), this implies that
\begin{equation}
\ul'(t)^{2}+\lambda^{2}\ul(t)^{2}\leq
\left\{\frac{2}{t_{0}^{2}}\ul(t_{0})^{2}+4\ul'(t_{0})^{2}\right\}\left(\frac{t_{0}}{t}\right)^{\mu},
%\qquad
%\forall t\in\left[\min\left\{t_{1},\frac{1}{\lambda}\right\},\frac{1}{\lambda}\right],
\nonumber
\end{equation}
which is enough to establish (\ref{th:prop-gen}) in the first time-interval.

Let us consider now the second interval, in the case where it is non-degenerate, namely $t_{1}<1/\lambda$. In this case we can exploit estimate (\ref{th:P-after}), from which we deduce that
\begin{eqnarray*}
\ul'(t)^{2}+\lambda^{2}\ul(t)^{2} & \leq &
2\lambda^{2}\ul(t_{0})^{2}+2\ul'(t_{0})^{2}\lambda^{2}\gamma(m,t_{0},t_{1})^{2}
\\[0.5ex]
& \leq &
2\lambda^{2}\ul(t_{0})^{2}+2\ul'(t_{0})^{2}\lambda^{2}\gamma(m,t_{0},t)^{2}.
\end{eqnarray*}

Since also in this interval we know that $t\leq 1/\lambda$, namely $\lambda\leq 1/t$, recalling (\ref{th:gamma}) we deduce that
\begin{eqnarray*}
\ul'(t)^{2}+\lambda^{2}\ul(t)^{2} & \leq &
\frac{2}{t_{0}^{2}}\left(\frac{t_{0}}{t}\right)^{2}\ul(t_{0})^{2}+
2\ul'(t_{0})^{2}\frac{1}{t_{0}^{2}}\left(\frac{t_{0}}{t}\right)^{2}\gamma(m,t_{0},t)^{2}
\\[0.5ex]
& \leq &
\frac{2}{t_{0}^{2}}\left(\frac{t_{0}}{t}\right)^{2}\ul(t_{0})^{2}+
2\ul'(t_{0})^{2}\left(\frac{t_{0}}{t}\right)^{\mu}
\\[0.5ex]
& \leq &
\left\{\frac{2}{t_{0}^{2}}\ul(t_{0})^{2}+2\ul'(t_{0})^{2}\right\}\left(\frac{t_{0}}{t}\right)^{\mu},
\end{eqnarray*}
which proves (\ref{th:prop-gen}) also in the second time-interval. 

Finally, let us consider the half-line $t\geq 1/\lambda$. When $t=1/\lambda$ the last estimate tells us that
\begin{equation}
\ul'\left(\frac{1}{\lambda}\right)^{2}+\lambda^{2}\ul\left(\frac{1}{\lambda}\right)^{2}\leq
\left\{\frac{2}{t_{0}^{2}}\ul(t_{0})^{2}+2\ul'(t_{0})^{2}\right\}(\lambda t_{0})^{\mu}.
\nonumber
\end{equation}

For $t\geq 1/\lambda$ we apply again Proposition~\ref{prop:hyperbolic}, but now with initial time $1/\lambda$ instead of $t_{0}$, and from estimate (\ref{th:hyperbolic}) (with $1/\lambda$ instead of $t_{0}$) we deduce that
\begin{eqnarray*}
\ul'(t)^{2}+\lambda^{2}\ul(t)^{2} & \leq &
e^{m(M+8)}
\left\{\ul'\left(\frac{1}{\lambda}\right)^{2}+\lambda^{2}\ul\left(\frac{1}{\lambda}\right)^{2}\right\}
\left(\frac{1}{\lambda t}\right)^{m}
%\exp\left(m(M+8)\right)
\\[0.5ex]
& \leq &
e^{m(M+8)}
\left\{\frac{2}{t_{0}^{2}}\ul(t_{0})^{2}+2\ul'(t_{0})^{2}\right\}(\lambda t_{0})^{\mu}
\left(\frac{1}{\lambda t}\right)^{m}
%\exp\left(m(M+8)\right)
\\[0.5ex]
& \leq &
e^{m(M+8)}
\left\{\frac{2}{t_{0}^{2}}\ul(t_{0})^{2}+2\ul'(t_{0})^{2}\right\}
\left(\frac{t_{0}}{t}\right)^{\mu},
%\exp\left(m(M+8)\right),
\end{eqnarray*}
which proves (\ref{th:prop-gen}) also in the last half-line.
\qed

%\clearpage

\subsubsection*{Statement~(2)}

To begin with, we observe that the coefficient $b(t)$ defined by (\ref{defn:b-hyp-alpha}) satisfies both the assumptions of Proposition~\ref{prop:hyp-alpha}, and the assumption of Proposition~\ref{prop:parabolic} with
\begin{equation}
b_{1}(t):=\frac{a}{t},
\qquad\quad
b_{2}(t):=\frac{r\sin(t^{\alpha})}{t},
\qquad\quad
m:=a,
\qquad\quad
B:=\frac{3r}{\alpha t_{0}^{\alpha}}
\nonumber
\end{equation}
(the verification of assumption (\ref{hp:b2}) is the same elementary computation that leads to (\ref{est:pre-b2}) in the proof of Proposition~\ref{prop:hyp-alpha}).

From now on we proceed exactly as in the proof of statement~(1), with the only difference that now $B>0$. If $\lambda\geq 1/t_{0}$ we apply Proposition~\ref{prop:hyp-alpha}, and from (\ref{th:hyp-alpha}) we obtain that
\begin{equation}
\ul'(t)^{2}+\lambda^{2}\ul(t)^{2}\leq
\Gamma_{4}
\left\{\ul'(t_{0})^{2}+\lambda^{2}\ul(t_{0})^{2}\right\}\left(\frac{t_{0}}{t}\right)^{m}
\qquad
\forall t\geq t_{0},
\nonumber
\end{equation}
which implies (\ref{th:prop-alpha}) because when $\lambda\geq 1/t_{0}$ the constant $\Gamma_{4}$ defined by (\ref{defn:Gamma-bis}) is less than the constant $\Gamma_{2}$ defined by (\ref{defn:Gamma}).

If $\lambda<1/t_{0}$ we divide the half-line $t\geq t_{0}$ into the three subsets (\ref{division}). In the first one we obtain that 
\begin{equation}
\ul'(t)^{2}+\lambda^{2}\ul(t)^{2}\leq
\left\{\frac{2}{t_{0}^{2}}\ul(t_{0})^{2}+4e^{2B}\ul'(t_{0})^{2}\right\}\left(\frac{t_{0}}{t}\right)^{\mu},
%\qquad
%\forall t\in\left[\min\left\{t_{1},\frac{1}{\lambda}\right\},\frac{1}{\lambda}\right],
\nonumber
\end{equation}
which is enough to establish (\ref{th:prop-alpha}) in the first time-interval.

In the second interval we obtain that
\begin{equation}
\ul'(t)^{2}+\lambda^{2}\ul(t)^{2}\leq
\left\{\frac{2}{t_{0}^{2}}\ul(t_{0})^{2}+2e^{2B}\ul'(t_{0})^{2}\right\}\left(\frac{t_{0}}{t}\right)^{\mu},
\nonumber
\end{equation}
which proves (\ref{th:prop-alpha}) also in the second time-interval. 

Finally, in the half-line $t\geq 1/\lambda$ we apply again Proposition~\ref{prop:hyp-alpha}, but now with initial time $1/\lambda$ instead of $t_{0}$, and from estimate (\ref{th:hyp-alpha}) (with $1/\lambda$ instead of $t_{0}$) we deduce that
\begin{equation}
\ul'(t)^{2}+\lambda^{2}\ul(t)^{2}\leq
\Gamma_{2}
\left\{\frac{2}{t_{0}^{2}}\ul(t_{0})^{2}+2e^{2B}\ul'(t_{0})^{2}\right\}
\left(\frac{t_{0}}{t}\right)^{\mu},
%\Gamma,
\nonumber
\end{equation}
which proves (\ref{th:prop-alpha}) also in the last half-line.
\qed

%\clearpage

\subsubsection*{Statement~(3)}

\paragraph{\textmd{\textit{Definition of the damping coefficient}}}

Let $\etal:(0,+\infty)\to\re$ denote the solution to the ordinary differential equation
\begin{equation}
\etal'(t)=\lambda-\frac{a+r\cos(2\etal(t))}{2t}\sin(2\etal(t))
\qquad
\forall t>0,
\label{eqn:ode-eta}
\end{equation}
with ``initial'' condition
\begin{equation}
\etal(t_{0})=\frac{\pi}{2}.
\nonumber
\end{equation}

We claim that the conclusions hold true if we set
\begin{equation}
b(t):=\frac{a+r\cos(2\etal(t))}{t}
\qquad
\forall t>0.
\label{defn:b}
\end{equation}

To this end, we observe first that $\etal(t)$ can be written in the form
\begin{equation}
\etal(t)=\lambda t+\hl(t),
\label{eqn:eta-h}
\end{equation}
for a suitable function $\hl:(0,+\infty)\to\re$ that satisfies
\begin{equation}
|\hl'(t)|\leq\frac{a+r}{2t}
\qquad
\forall t>0.
\label{est:h'-eta}
\end{equation}

\paragraph{\textmd{\textit{Scale invariant behavior and integrability of oscillations}}}

The pointwise bounds (\ref{th:bound-b}) are automatic from definition (\ref{defn:b}). 

In order to prove (\ref{th:lim-B}) we observe that
\begin{equation}
\int_{t_{0}}^{t}b(s)\,ds=
a\log\left(\frac{t}{t_{0}}\right)+r\int_{t_{0}}^{t}\frac{\cos(2\etal(s))}{s}\,ds.
\nonumber
\end{equation}

Thanks to (\ref{eqn:eta-h}) and (\ref{est:h'-eta}), we can apply Lemma~\ref{lemma:osc-int} with $n=2$ and conclude that
\begin{equation}
\lim_{t\to +\infty}\left(\frac{t_{0}}{t}\right)^{a}\exp\left(\int_{t_{0}}^{t}b(s)\,ds\right)=
\lim_{t\to +\infty}\exp\left(r\int_{t_{0}}^{t}\frac{\cos(2\etal(s))}{s}\,ds\right)
\nonumber
\end{equation}
exists and is a positive real number. 

\paragraph{\textmd{\textit{Slower decay of one solution}}}

Let us consider the solution to equation (\ref{ode:eqn}) with initial data
\begin{equation}
\ul(t_{0})=0,
\qquad\qquad
\ul'(t_{0})=1,
\label{ode:data}
\end{equation}
and let us write it in the form (\ref{defn:polar}). In this way we reduce ourselves to estimating from below the exponential in (\ref{est:rho2}). To this end, we observe that now equation (\ref{eqn:theta}) for $\thetal$ reads as
\begin{equation}
\thetal'(t)=\lambda-\frac{1}{2}\frac{a+r\cos(2\etal(t))}{t}\sin(2\thetal(t)),
\qquad\quad
\thetal(t_{0})=\frac{\pi}{2}.
\nonumber
\end{equation}

Comparing with (\ref{eqn:ode-eta}), by uniqueness we deduce that $\thetal(t)=\etal(t)$ for every $t>0$. Now from (\ref{defn:b}) with some trigonometry we deduce that
\begin{equation}
-2b(s)\sin^{2}(\thetal(s))=
-\left(a-\frac{r}{2}\right)\frac{1}{s}+
(a-r)\frac{\cos(2\etal(s))}{s}+
\frac{r}{2}\cdot\frac{\cos(4\etal(s))}{s},
\nonumber
\end{equation}
and therefore 
\begin{multline*}
\qquad
-2\int_{t_{0}}^{t}b(s)\sin^{2}(\thetal(s))\,ds=
\left(a-\frac{r}{2}\right)\log\left(\frac{t_{0}}{t}\right)
\\[0.5ex]
+(a-r)\int_{t_{0}}^{t}\frac{\cos(2\etal(s))}{s}\,ds+
\frac{r}{2}\int_{t_{0}}^{t}\frac{\cos(4\etal(s))}{s}\,ds.
\qquad
\end{multline*}

Thanks again to (\ref{eqn:eta-h}) and (\ref{est:h'-eta}), we can apply Lemma~\ref{lemma:osc-int} with $n=2$ and $n=4$, and conclude from (\ref{th:osc-int-bound}) that the last two integrals are bounded from below (and also from above). This completes the proof of (\ref{th:resonance}).
\qed

%\clearpage

\setcounter{equation}{0}
\section{From ODEs to PDEs (proof of main results)}\label{sec:PDEs}

\subsubsection*{Proof of Theorem~\ref{thm:decay}}

The argument is rather standard. We identify $A$ with the multiplication operator by $\lambda(\xi)^{2}$ in $L^{2}(\mathcal{M},\mu)$, then for every $\xi\in\mathcal{M}$ we consider the generalized Fourier transform  $\ut(t,\xi):=[\mathscr{F}u(t)](\xi)$ of the solution to (\ref{eqn:main})--(\ref{eqn:data}), and we recall that for every $\xi\in M$ it is a solution to problem (\ref{eqn:uxi})--(\ref{data:xi}).

Now we apply statement~(1) of Proposition~\ref{prop:main} with $\lambda:=\lambda(\xi)$, and we deduce that
\begin{equation}
\ut\,'(t,\xi)^{2}+\lambda(\xi)^{2}\ut(t,\xi)^{2}\leq
e^{m(M+8)}\left\{4\ut\,'(t_{0},\xi)^{2}+\left(\lambda^{2}(\xi)+\frac{2}{t_{0}^{2}}\right)\ut(t_{0},\xi)^{2}\right\}
\left(\frac{t_{0}}{t}\right)^{\mu}
\nonumber
\end{equation}
for every $\xi\in\mathcal{M}$ and every $t\geq t_{0}$. Recalling (\ref{defn:norm-xi}) and (\ref{defn:A-alpha}), when we integrate with respect to $\xi$ we obtain exactly (\ref{th:main-decay}).
\qed

\subsubsection*{Proof of Theorem~\ref{thm:alpha}}

The argument is analogous to the proof of Theorem~\ref{thm:decay}, just with statement~(2) of Proposition~\ref{prop:main} instead of statement~(1).
\qed

\subsubsection*{Proof of Theorem~\ref{thm:resonance}}

For every pair of positive real numbers $\lambda$ and $s$ we consider the set
\begin{equation}
\mathcal{M}_{\lambda,s}:=\{\xi\in \mathcal{M}:|\lambda(\xi)-\lambda|\leq s\}.
\nonumber
\end{equation}

Since $A$ is not identically zero, there exists a positive real number $\lambda_{0}$ such that
\begin{equation}
\mu(\mathcal{M}_{\lambda_{0},s})>0
\qquad
\forall s>0,
\nonumber
\end{equation}
and we consider the damping coefficient $b(t)$ provided by statement~(3) of Proposition~\ref{prop:main} with $\lambda:=\lambda_{0}$. For every $\lambda>0$ we consider the solution to (\ref{ode:eqn}) with this choice of $b(t)$ and initial data (\ref{ode:data}).  
For every fixed $t\geq t_{0}$ we know from (\ref{th:resonance}) that
\begin{equation}
u_{\lambda_{0}}'(t)^{2}+\lambda_{0}^{2}u_{\lambda_{0}}(t)^{2}\geq
\Gamma_{3}\left(\frac{t_{0}}{t}\right)^{a-r/2},
\nonumber
\end{equation}
where $\Gamma_{3}$ is the constant that appears in statement~(3) of Proposition~\ref{prop:main}, and depends only on $t_{0}$, $a$, $r$, $\lambda_{0}$. Since solutions to (\ref{ode:eqn})--(\ref{ode:data}) depend continuously on $\lambda$, we deduce that there exists $s>0$ (that depends on $t$) such that
\begin{equation}
\ul'(t)^{2}+\lambda^{2}\ul(t)^{2}\geq
\frac{\Gamma_{3}}{2}\left(\frac{t_{0}}{t}\right)^{a-r/2}
\qquad
\forall\lambda\in[\lambda_{0}-s,\lambda_{0}+s],
\nonumber
\end{equation}
and in particular the solution to (\ref{eqn:uxi}) with initial data
\begin{equation}
\ut(t_{0},\xi)=0,
\qquad\qquad
\ut\,'(t_{0},\xi)=1
\nonumber
\end{equation}
satisfies
\begin{equation}
\ut\,'(t,\xi)^{2}+\lambda(\xi)^{2}\ut(t,\xi)^{2}\geq
\frac{\Gamma_{3}}{2}\left(\frac{t_{0}}{t}\right)^{a-r/2}
\qquad
\forall\xi\in\mathcal{M}_{\lambda_{0},s}.
\nonumber
\end{equation}

At this point we can set
\begin{equation}
\ut_{1}(\xi):=\begin{cases}
\dfrac{1}{\mu(\mathcal{M}_{\lambda,s})}      & \text{if }\xi\in\mathcal{M}_{\lambda,s}, 
\\[2ex]
0      & \text{otherwise},
\end{cases}
\nonumber
\end{equation}
and conclude that the solution to (\ref{eqn:main}) with initial data $u(t_{0})=0$ and $u'(t_{0})=\mathscr{F}^{-1}\ut_{1}$ satisfies
\begin{equation}
|u'(t_{0})|^{2}+|A^{1/2}u(t_{0})|^{2}=1
\qquad\text{and}\qquad
|u'(t)|^{2}+|A^{1/2}u(t)|^{2}\geq
\frac{\Gamma_{3}}{2}\left(\frac{t_{0}}{t}\right)^{a-r/2}.
\nonumber
\end{equation}

This is enough to conclude that (\ref{th:bad-decay}) holds true for this fixed value of $t$.
\qed

%\clearpage

\subsubsection*{\centering Acknowledgments}

Both authors are members of the Italian {\selectlanguage{italian}% 
``Gruppo Nazionale per l'Analisi Matematica, la Probabilit\`{a} e le loro Applicazioni'' (GNAMPA) of the ``Istituto Nazionale di Alta Matematica'' (INdAM)}. The authors acknowledge also  the MIUR Excellence Department Project awarded to the Department of Mathematics, University of Pisa, CUP I57G22000700001. The first author was partially supported by PRIN 2020XB3EFL, ``Hamiltonian and Dispersive PDEs''.

\selectlanguage{english}

%%%\clearpage

%{\small 
% \bibliographystyle{../../../BibTeX/MaxNew}
% \bibliography{../../../BibTeX/Damping}
% \bibliographystyle{MaxNew}
% \bibliography{Damping}
%}

\label{NumeroPagine}

\end{document}